\documentclass[preprint,11pt]{elsarticle}
\usepackage{amssymb}
\usepackage{amsfonts,latexsym}
\usepackage{epsfig}
\usepackage{amsmath}
\usepackage{amsthm}
\usepackage{epstopdf}
\usepackage{graphicx}
\usepackage{amsmath}
\usepackage{color}
\usepackage{ulem}
\usepackage{tikz}
\usepackage{algorithm}
\usepackage{algpseudocode}
\usepackage{algorithmicx}
\usepackage{booktabs}

\usepackage{mathtools}
\usepackage[scr=boondoxupr]{mathalpha}
\newcommand{\rightangle}{\mathrel{\text{\tikz \draw[baseline=-2.9ex] (0,0) -- (0,1.2ex) -- (1.2ex,1.2ex);}}}
\usetikzlibrary{decorations.markings}
\usepackage[top=3cm,bottom=4cm,right=3cm,left=3cm]{geometry}

\journal{Elsevier}

\begin{document}

\begin{frontmatter}

\title{ Novel Approach for solving the discrete Stokes problems based on Augmented Lagrangian and Global Techniques with applications for Stokes problems }

\author[1]{A. Badahmane}
\author[1]{A. Ratnani}
\author[2]{H. Sadok}
\address[1]{The UM6P Vanguard Center, Mohammed VI Polytechnic University, Benguerir 43150, Lot 660, Hay Moulay Rachid, Morocco.}
\address[2]{LMPA, Universit\'e du Littoral C\^ote d'Opale, 50 Rue F. Buisson, BP 699 - 62228 Calais cedex, France.}

%\author[1]{A. Badahmane}
%\author[2]{H. Sadok}
%\address[1]{The UM6P Vanguard Center, Mohammed VI Polytechnic University, Benguerir 43150, Lot 660, Hay Moulay Rachid, Morocco}
%\address[1]{The UM6P Vanguard Center, Mohammed VI Polytechnic University, Benguerir 43150, Lot 660, Hay Moulay Rachid, Morocco.}

%\address[2]{LMPA, Universit\'e du Littoral C\^ote d'Opale, 50 Rue F. Buisson, BP 699 - 62228 Calais cedex, France, email : hassane.sadok@univ-littoral.fr.}

\newcommand{\Leg}{{\mathcal L(E,G)}}
\newcommand{\Lef}{{\mathcal L(E,F)}}
\newcommand{\Lfg}{{\mathcal L(F,G)}}
\newcommand{\Le}{{\mathcal L(E)}}
\newcommand{\I}{{\mathcal{I}}}
\newcommand{\ia}{{\mathfrak I}}
\newcommand{\vi}{\emptyset}
\newcommand{\di}{\displaystyle}
\newcommand{\Om}{\Omega}
\newcommand{\na}{\nabla}
\newcommand{\wi}{\widetilde}
\newcommand{\al}{\alpha}
\newcommand{\be}{\beta}
\newcommand{\ga}{\gamma}
\newcommand{\Ga}{\Gamma}
\newcommand{\e}{\epsilon}
\newcommand{\la}{\lambda}
\newcommand{\De}{\Delta}
\newcommand{\de}{\delta}
\newcommand{\entraine}{\Longrightarrow}
\newcommand{\inj}{\hookrightarrow}
\newcommand{\recip}{\Longleftarrow}
\newcommand{\ssi}{\Longleftrightarrow}
\newcommand{\K}{\mathbbm{K}}
\newcommand{\A}{\mathcal{A}}
\newcommand{\R}{\mathbb{R}}
\newcommand{\C}{\mathbb{C}}
\newcommand{\N}{\mathbb{N}}
\newcommand{\Q}{\mathbb{Q}}
\newcommand{\1}{\mathbb{1}}
\newcommand{\0}{\mathbb{0}}
\newcommand{\Z}{\mathbbm{Z}}
\newcommand{\E}{\mathbbm{E}}
\newcommand{\F}{\mathbbm{F}}
\newcommand{\B}{\mathbbm{B}}
\newcommand{\M}{\mathcal{M}_{n}(\K)}

\font\bb=msbm10

\def\ent{{{\rm Z}\mkern-5.5mu{\rm Z}}}
\newtheorem{exo}{Exercice}
\newtheorem{idea}{Idea}

\newtheorem{pre}{Preuve}
\newtheorem{pro}{Propriété}
\newtheorem{exe}{Example}
\newtheorem{theorem}{Theorem}[section]
\newtheorem{proposition}{Proposition}
\newtheorem{definition}{Definition}[section]
\newtheorem{remark}{Remark}[section]
\newtheorem{lem}{Lemma}[section]

\begin{abstract}
In this paper, a novel   augmented Lagrangian preconditioner based on global Arnoldi for  accelerating the convergence of Krylov subspace methods applied to linear systems of equations with a block three-by-three structure, these systems typically arise from  discretizing  the Stokes equations using mixed-finite element methods. In practice, the components of velocity are always approximated using a single finite element space. More precisely, in two dimensions, our new approach based on  standard space of scalar finite element basis functions to discretize the velocity space. This componentwise splitting can be shown to induce  a natural block  three-by-three structure.
Spectral analyses is established for the exact versions of these preconditioners. Finally, the obtained numerical results claim that our  novel approach is more efficient and robust  for solving the  discrete Stokes problems. The efficiency of our new approach is evaluated by measuring computational time.
\end{abstract}
\end{frontmatter}

\section{Introduction}\label{sec1}
 The Stokes problem is discretized using conforming finite element spaces \( X^{h} \subset Q_2 \) and \( Q^{h}_{1} \subset Q_1 \) that satisfy the inf-sup condition for the Stokes velocity and pressure such as  Taylor–Hood elements \cite{Elman}.
The discrete form of the weak formulation can be cast as a block linear system of the form:
\begin{eqnarray}
\label{saddle}
\mathcal{A}_{3\times 3}{\mathbf{u} }=\begin{bmatrix}
A & O & B_x^T \\
O & A & B_y^T \\
B_x & B_y & O
\end{bmatrix}
\begin{bmatrix}
u_x \\
u_y \\
p
\end{bmatrix}
=
\underbrace{\begin{bmatrix}
f_x \\
f_y \\
g
\end{bmatrix}}_{b}, 
\end{eqnarray}
assuming that  $n_u=2n$ and $n_p$ are  respectively,  dimensions of velocity solution and  pressure finite-dimensional spaces with $(n_{u}+n_{p}=N)$. 
Where  $A\in\mathbb{R}^{n\times n}$ is the scalar-Laplacian matrix, it is worth nothing  that $A$  is  symmetric positive definite (SPD) matrix, the \( n_p \times n \) matrices \( B_x \) and \( B_y \) represent weak derivatives in the \( x \) and \( y \) directions,
$f_x$, $f_y$ and  $g$ are given vectors. Besides, we assume, as is typically the case in most applications of Stokes problem, $n_{u}>>n_{p}$. The increasing popularity of mixed finite element methods for Stokes and Navier-Stokes flows has been a significant cause of saddle-point systems, such as the one in (\ref{saddle}). A major source of applications for saddle-point problems, can be found in \cite{Benzi2,Elman}. 
In general, inasmuch as the large dimension and sparsity of the matrices $A$ and  $B$,
it is sensible for systems (\ref{saddle}) to be solved by iterative methods. Additionally, since the
coefficient matrix A is nonsingular, in recent years, several eﬀective methods have been developed to
tackle the systems (\ref{saddle}). Such as the successive overrelaxation (SOR)-like methods \cite{Parlett, Wang, Golub2,Guo}, variants of the Uzawa-type methods \cite{Parlett, Wang,Bramble,Zhang},
Hermitian and skew-Hermitian (HSS) method, which was initially  introduced by Bai, Golub, and Ng in \cite{HSS}. Additionally, the  PHSS iteration method has been presented in \cite{PHSS}. For a more in-depth understanding of the works related to the stationary iterative
methods, please refer to references \cite{PHSS,HSS,AHSS}.
Generally speaking, iteration methods become more attractive than direct methods from
two aspects of storage requirements and computing time. In order to solve the linear system (\ref{saddle}) in an efficient manner, we often use valid preconditioning techniques to accelerate
Krylov subspace methods, such as GMRES  method \cite{Saad}. As is well known, a clustered spectrum of preconditioned matrix often results in rapid rate of convergence for Krylov
subspace methods. Therefore, to achieve rapid convergence rate and improve computational
efficiency, a large number of efficient iteration methods and preconditioning techniques have
been presented in recent years, such as block triangular preconditioner applied  to the augmented linear system \cite{beik2022}, augmented Lagrangian-based preconditioning technique for a class of block three-by-three linear systems \cite{benzi2024}, and so forth. We make organizations of this paper as follows. An example of modelling that leads to this
type of system is outlined in Section $1$.
Section $2$ introduces the $3\times3$ strategy. In Section $3$, we recall and define the $2\times 2$ strategy.
Some numerical tests are implemented to show the eﬀectiveness of  the proposed  preconditioners, in particular in the presence of inexact solvers. At the end, we conclude
with a brief summary in Section $5$.
\subsection{The Problem Setting}~\label{problem}
The Stokes equation describes the flow of a viscous fluid and is used in various fields, including aerodynamics, propulsion, and biomedical fluid analysis. In many cases, finding an exact solution to the Stokes equation can be challenging, so we often use numerical methods to approximate the solution~\cite{Elman}. Their discretization results in a linear system, as shown in Eq.~(\ref{saddle}). In the incompressible case, the Stokes equation can be written as follows :
\begin{eqnarray}
\label{eq:Stokes}
\left\{
\begin{array}{r c l}
-\vec{\nabla}^{2} \vec{u} +\vec{\nabla} p &=& \vec{0} \hspace{0.1cm} in  \hspace{0.1cm}\Omega,\\
%\displaystyle{\vec{\nabla}\cdot \vec{u} }&=& 0\hspace{0.1cm}\text{ in }\hspace{0.1cm} \Omega.
\vec{\nabla}\cdot \vec{u} &=& 0 \hspace{0.1cm} in \hspace{0.1cm} \Omega.
\end{array}
\right.
\end{eqnarray}
The variable $\vec{u}$ is the unknown velocity field, the scalar function  $p$ is the unknown pressure field. It is important to acknowledge that the Laplacian  and divergence operators are defined in~\cite{Elman}.
The first equation in  Eq.~(\ref{eq:Stokes}) represents conservation of the momentum of the fluid (and so is the momentum equation), and the second equation enforces conservation of mass. We consider the problem posed on a domain $\Omega$ of dimension $d=2$ with boundary conditions $\partial{\Omega}=\partial\Omega_{D}  \cup \partial \Omega_{N}$ defined by
\begin{eqnarray}
\label{Neumann}
\begin{array}{r c l}
\vec{u}&=& \vec{w} \text{ on } \partial\Omega_{D},\hspace{0.2cm}\displaystyle{\frac{\partial\vec{u}}{\partial n}-\vec{n}p}= \vec{s}  \text{ on } \partial\Omega_{N},
\end{array}
\end{eqnarray}
where :
\begin{itemize}
    \item  $\vec{w}$: is the vorticity variable, given by:
\begin{eqnarray}
\label{vor}
\vec{w}&=&\vec{\nabla}\times\vec{u},
\end{eqnarray}
where $\times$ is the curl operator,
%\item  $\partial\Omega_{N}$: Neumann  boundary condition,
\item $\vec{s}$: function depends on  the outflow boundary to ensure that mass is conserved,  
\item $\vec{n}$: the outward-pointing normal to the boundary,
\item $\displaystyle{\frac{\partial\vec{u}}{\partial n}}$: denotes
the directional derivative in the normal direction.
\end{itemize}
Before starting the weak formulation of the Stokes problem Eq.~(\ref{eq:Stokes}), we provide some definitions and reminders:\\
The space of functions that are
square-integrable according to  Lebesgue definition is a set of functions  where the integral of the square of the function over a given interval is finite, and also can, be expressed as follows:
$$L_{2}(\Omega):=\left\{e:\Omega \rightarrow\mathbb{R} \hspace{0.1cm}
 |\int_{\Omega}e^{2}<\infty\right\},$$
if we have a subset $\Omega$ of the two-dimensional Euclidean space $\mathbb{R}^2$, then the Sobolev space $\mathcal{H}^{1}(\Omega)$ can be defined as follows
$$\mathcal{H}^{1}(\Omega):=\left\{e:\Omega \rightarrow  \mathbb{R}\hspace{0.1cm}|e, \frac{\partial e}{\partial x},\frac{\partial e}{\partial y}\in L_{2}(\Omega)\right\}.$$ 
We define the velocity solution and test spaces:
$$H^{1}_{E}:=\left\{\vec{u}\in \mathcal{H}^{1}(\Omega)^{d}\hspace{0.1cm}|  \vec{u}=\vec{w} \hspace{0.1cm} \text{on} \hspace{0.1cm}\partial\Omega_{D}\right\},$$
$$H^{1}_{E_0}:=\left\{\vec{v}\in \mathcal{H}^{1}(\Omega)^{d}\hspace{0.1cm}|  \vec{v}=\vec{0} \hspace{0.1cm} \text{on} \hspace{0.1cm}\partial\Omega_{D}\right\},$$
where $d=2$  is the spatial dimension. The variational formulation of  
$(\ref{eq:Stokes})$,
find  $\vec{u}\in H^{1}_{E}$ and $p\in L_{2}(\Omega)$ such that :
\begin{eqnarray}
\label{var}
\begin{array}{r c l}
\displaystyle{\int_{\Omega }^{}\vec{\nabla} \vec{u}:\vec{\nabla} \vec{v}-\int_{\Omega }^{}p\vec{\nabla}\cdot \vec{v} }&=&\displaystyle{ \int_{\Omega }^{}\vec{s}\cdot  \vec{v}}\hspace{0.2cm}\  \text{for all}\hspace{0.1cm} \vec{v}\in H^{1}_{E_0}, \\
 \displaystyle{\int_{\Omega }^{}q\vec{\nabla}\cdot \vec{u}}&=& 0 \hspace{1.6cm} \text{for all}\hspace{0.1cm} q\in L_{2}(\Omega).
\end{array}
\end{eqnarray}
Here  $\cdot$ is the scalar product and  $\vec{\nabla} \vec{u} : \vec{\nabla} \vec{v}$ represents the component-wise scalar product. For instance, in two dimensions, it can be represented as $\vec{\nabla} u_{x}\cdot\vec{\nabla} v_{x}+\vec{\nabla} u_{y}\cdot\vec{\nabla} v_{y}$. A discrete weak formulation is defined
using   finite dimensional spaces $X^{h}_{0}\subset H^{1}_{E_0}$ and $Q^{h}\subset L_{2}(\Omega)$ are respectively velocity solution   finite $n_{u}$-dimensional space  and pressure finite $n_{p}$-dimensional  space.  Specifically, given a velocity solution space $X^{h}$, the discrete form of $(\ref{var})$  is defined as follows : 
find $\vec{u}_\text{h}\in X^\text{h}$ and $p_\text{h}\in Q^\text{h}$, such that:
 \begin{eqnarray}
 \label{disc_var}
 \begin{array}{r c l}
 \displaystyle{\int_{\Omega }^{}\vec{\nabla}\vec{u}_\text{h}:\vec{\nabla}\vec{v}_\text{h}-\int_{\Omega }^{}p_\text{h}\vec{\nabla}\cdot \vec{v}_\text{h}} &=& \displaystyle{\int_{\Omega }^{}\vec{s}\cdot  \vec{v}_\text{h}}\hspace{0.2cm}\  \text{for all}\hspace{0.1cm} \vec{v}_\text{h}\in X^\text{h}_{0}, \\
 \displaystyle{\int_{\Omega }^{}q_\text{h}\vec{\nabla} \cdot\vec{u}_\text{h}}&=&0 \hspace{1.6cm}\text{for all}\hspace{0.1cm} {q}_\text{h}\in Q^\text{h}. \\ 
 \end{array}
 \end{eqnarray}
To identify the corresponding linear algebra problem  Eq.~(\ref{saddle}), we introduce a set of vector-valued basis functions $\{\vec{\phi}_{j}\}_{j=1,...,n_{u}}$, that represent velocity  and  a set of
scalar pressure basis functions $\{\psi_{k}\}_{k=1,...,n_{p}}$, for more details we refer the reader to see \cite{Elman}, then $\vec{u}_\text{h}$ and $p_\text{h}$
can be expressed as follows:
\begin{eqnarray}
\label{vilos}
\vec{u}_\text{h}=\sum_{j=1}^{n_{u}} \textbf{u}_{j} \vec{\phi_{j}} +\sum_{j=n_{u}+1}^{n_{u}+n_{\partial}} \textbf{u}_{j} \vec{\phi_{j}},  \hspace {0.3cm}p_\text{h}=\sum_{k=1}^{n_{p}} \textbf{p}_{k}\psi_{k},
\end{eqnarray}
and use them to formulate the problem in terms of linear algebra.
The discrete formulation Eq.~(\ref{disc_var}), can be expressed as a system of linear equations.
In practice, the $d$ components of velocity are always approximated using
a single finite element space \cite{Elman}, then the discrete formulation of Eq. 
 (\ref{eq:Stokes}) can be expressed as a two-by-two partitioning of the discrete Stokes system, which the matrix of the system is a saddle point matrix  defined as follows :
\begin{eqnarray}
\label{saddle2}
\mathcal{A}_{2\times2}x=\left(\begin{array}{cc}
A & B^{T} \\B & 0
\end{array} \right)\left(\begin{array}{c}
 \mathrm{u}\\ \mathrm{p}
\end{array} \right)=\underbrace{\left(\begin{array}{c}
\mathrm{f} \\ \mathrm{g}
\end{array} \right)}_{b},
\end{eqnarray}
where  $A_{2\times 2}\in\mathbb{R}^{n_{u}\times n_{u}}$ is the vector-Laplacian matrix, it is worth nothing  that $A$  is  symmetric positive definite (SPD) matrix, $B\in\mathbb{R}^{n_{p}\times n_{u}}$ is divergence matrix  with $\text{rank}(B^{T})=n_{p}$, $f\in\mathbb{R}^{n_{u}}$ and $g\in\mathbb{R}^{n_{p}}$ are given vectors.
where   $A$ and  $B$ are given by 
 \begin{eqnarray}
 A&=&[a_{i,j}],\hspace{0.2cm} a_{i,j}=\int_{\Omega} \vec{\nabla}\vec{\phi_{i}}:\vec{\nabla}\vec{\phi_{j}}, \hspace{0.2cm} i,j=1,..,n_{u}, \\
B&=&[b_{k,j}],\hspace{0.2cm} b_{k,j}=-\int_{\Omega} \psi_{k} \vec{\nabla}\cdot\vec{\phi_{j}}, \hspace{0.2cm} j=1,..,n_{u}, k=1,..,n_{p}.
 \end{eqnarray}
The right-hand side of the discrete Stokes problem can be expressed as follows: 
\begin{eqnarray}
  \mathrm{f}&=&[\mathrm{f}_{i}],\hspace{0.4cm} \mathrm{f}_{i}=\int_{\partial\Omega_{N}} \vec{s} \cdot\vec{\phi_{i}}-\sum_{j=n_{u}+1}^{n_{u}+n_{\partial}}\mathrm{u}_{j}\int_{\Omega} \vec{\nabla}\vec{\phi_{i}}:\vec{\nabla}\vec{\phi_{j}}, \hspace{0.2cm} i=1,..,n_{u}, \\
  \mathrm{g}&=&[\mathrm{g}_{k}],\hspace{0.1cm}  \mathrm{g}_{k}=\sum_{j=n_{u}+1}^{n_{u}+n_{\partial} }\mathrm{u}_{j}\int_{\Omega}\psi_{k}\vec{\nabla}\cdot\vec{\phi_{j}},\hspace{0.1cm}k=1,....,n_{p}.
  \end{eqnarray}
\section*{Motivation:}\label{idea}
The main motivation of this work, instead of using a single finite element space to discretize the velocity space and to obtain the two-by-two partitioning (\ref{saddle2}), we use a standard space of scalar finite element basis functions \(\{\phi_j\}_{j=1}^n\), we set \( n_u = 2n \) and define the velocity basis set
\[
\{\vec{\phi}_1, \ldots, \vec{\phi}_{2n}\} := \{(\phi_1, 0)^T, \ldots, (\phi_n, 0)^T, (0, \phi_1)^T, \ldots, (0, \phi_n)^T\}.
\]
This component-wise splitting can be shown to induce  a natural block three-by-three partitioning of the discrete Stokes system $(\ref{saddle})$, for more details, we refer to \cite{Elman}.
Specifically, with
\[
\textbf{u}:= ([u_x]_1, \ldots, [u_x]_n, [u_y]_1, \ldots, [u_y]_n),
\]
(\ref{saddle2}) can be rewritten as :
\[
\begin{bmatrix}
A & O & B_{x}^T \\
O & A & B_{y}^T \\
B_x & B_y & O
\end{bmatrix}
\begin{bmatrix}
u_x \\
u_y \\
p
\end{bmatrix}
=
\begin{bmatrix}
f_x \\
f_y \\
g
\end{bmatrix},
\]
where the \( n \times n \) matrix \( A \) is the scalar Laplacian matrix (discussed in detail in \cite{Elman}), and the \( n_p \times n \) matrices \( B_x \) and \( B_y \) represent weak derivatives in the \( x \) and \( y \) directions, 
where
 $$A = [a_{ij}], \hspace{0.2cm}a_{ij}=\int_{\Omega}\nabla\phi_{i}\cdot\nabla\phi_{j},$$
 $$B_{x} = [b_{x,kj}], \hspace{0.2cm}b_{x,ki}=-\int_{\Omega}\psi_{k}\frac{\partial\phi_{i}}{\partial x},$$
 $$B_{y} = [b_{y,kj}], \hspace{0.2cm}b_{y,kj}=-\int_{\Omega}\psi_{k}\frac{\partial\phi_{j}}{\partial y},$$
 where  $\{\psi_{k}\}_{k=1,...,n_{p}}$ a set of
scalar pressure basis functions, for more details we refer the reader to see \cite{Elman}.
 %and  B = [b_{kj}] , where  b_{kj} = - \psi \frac{\partial \phi_i}{\partial x}  and  B = [b_{kj}]  where  b_{kj} = - \psi \frac{\partial \phi_j}{\partial y}. 
\section*{Mathematical background:}
Given a square matrix $A$, the set of all eigenvalues (spectrum) of $A$ is denoted by $\sigma(A)$. When the spectrum of $A$ is real, we use $\lambda_{\min}(A)$ and $\lambda_{\max}(A)$ to respectively denote its minimum and maximum eigenvalues. When $A$ is symmetric positive (semi)definite, we write $A \succ 0$ ($A \succeq 0$). In addition, for two given matrices $A$ and $B$, the relation $A \succ B$ ($A \succeq B$) means $A - B \succ 0$ ($A - B \succeq 0$). Finally, for vectors $x$, $y$, and $z$ of dimensions $n$, $m$, and $p$, $(x; y; z)$ will denote a column vector of dimension $n+m+p$.
In this paper, $I$ will denote the identity matrix, specifying its size as appropriate to the context.

\section{$3\times3$  Strategy for solving three-by-three linear system (\ref{saddle})}
The $3\times3$ strategy, based on the motivation outlined in Section $1$, is designed to solve three-by-three saddle-point problem~(\ref{saddle}).
 The $3\times3$  strategy can significantly reduce the computational cost compared  using $2\times 2$  strategy for solving the classical structure of saddle-point problem (\ref{saddle2}). The preconditioning technique helps to improve the convergence rate of the Krylov subspace methods. 
This strategy is motivated by the use of a set of standard scalar finite element basis functions within a defined space, aimed at obtaining the three-by-three partitions of the saddle-point matrix (\ref{saddle}).

\subsection{Novel Augmented Lagrangian-based preconditioning and global
techniques:}
Krylov subspace methods (such as GMRES) in conjunction with suitable preconditioners are frequently the method of choice for computing approximate solutions of such linear systems of equations.
First, problem (\ref{saddle})  is reformulated as the equivalent augmented system $\bar{\mathcal{A}}_{3\times 3} \bar{\mathbf{u} }= \bar{\mathbf{b}}$, where
\begin{equation}
\bar{\mathcal{A}}_{3\times 3}= 
\begin{bmatrix}
A + \gamma B^T_x Q^{-1}B_x & 0 & B_x^T \\
0 & A + \gamma B^T_y Q^{-1}B_y & B^T_y \\
B_x & B_y & 0
\end{bmatrix},
\end{equation}
and $\bar{\mathbf{b}} = ({f}_x+B^T_x Q^{-1}g;{f}_y + \gamma B^T_y Q^{-1}g; g)$, with $Q$ being an arbitrary SPD matrix and $\gamma > 0$ a user-defined parameter. Evidently, the linear system of equations 
$\bar{\mathcal{A}}_{3\times 3} \bar{\mathbf{u} }= \bar{\mathbf{b}}$ is equivalent to ${\mathcal{A}}_{3\times 3}\mathbf{u} = {\mathbf{b}}$. This approach is inspired by the effectiveness of employing grad-div stabilization and augmented Lagrangian techniques to solve saddle-point problems.

\subsubsection{ Preconditioning:}\label{Pia}
In this section, we investigate a new augmented Lagrangian-based preconditioning and global approach for solving (\ref{saddle}).  Left preconditioning of  (\ref{saddle}) gives the following new linear system:
\begin{equation}
\label{aug}
\mathcal{P}^{-1}\bar{\mathcal{A}}_{3\times 3}\bar{\mathbf{u}} = \mathcal{P}^{-1}\bar{\mathbf{b}},
\end{equation}
where  $\mathcal{P}$ is one of the preconditioners  below:
\begin{itemize}
    \item 
$\mathcal{P}_{\gamma, \alpha, x} $ : is the  augmented Lagrangian preconditioner in the $x$ direction.
 \item $\mathcal{P}_{\gamma, \alpha, y} $ : is the  augmented Lagrangian preconditioner in the $y$ direction.
\end{itemize}
 The following two constraint-type preconditioners were proposed for accelerating the convergence of Krylov subspace methods, given as follows:
\begin{eqnarray}
\label{Pgamma}
\mathcal{P}_{\gamma, \alpha,x} = 
\begin{bmatrix}
A + \gamma B^T_x Q^{-1}B_x & 0 & B_x^T \\
0 & A + \gamma B^T_x Q^{-1}B_x & (1 - \gamma \alpha^{-1}) B^T_y \\
0 & 0 & -\alpha^{-1}Q
\end{bmatrix},
\end{eqnarray}

\begin{eqnarray}
\mathcal{P}_{\gamma, \alpha,y} = 
\begin{bmatrix}
A + \gamma B^T_y Q^{-1}B_y & 0 & B_x^T \\
0 & A + \gamma B^T_y Q^{-1}B_y & (1 - \gamma \alpha^{-1}) B^T_y \\
0 & 0 & -\alpha^{-1}Q
\end{bmatrix},
\end{eqnarray}
where $\alpha$ and $\gamma$ are prescribed positive parameters.
\subsubsection{Algorithmic implementation  of the augmented Lagrangian preconditioners  $\mathcal{P}_{\gamma, \alpha,x}$ and $\mathcal{P}_{\gamma, \alpha,y}$.}
In this part, we display the algorithmic implementation of $\mathcal{P}_{\gamma, \alpha,x}$ and $\mathcal{P}_{\gamma, \alpha,y}$,
in which, inside Krylov subspace methods, the SPD subsystems were solved inexactly by the preconditioned conjugate gradient (PCG) method using loose tolerances. More precisely, the inner PCG solver for linear systems with coefficient matrix $A$, $A + \gamma B^{T}_{x}Q^{-1} B_{x}$ and $A + \gamma B^{T}_{y}Q^{-1} B_{y}$ was terminated when the relative residual norm was below $10^{-6}$, with the maximum number of $100$ iterations was reached. The preconditioner for PCG is incomplete Cholesky factorizations constructed using the  function \texttt{ichol(., opts)} where \text{opts.type = 'ict'} with drop tolerance $10^{-2}$. 
In the following parts, we will work on some specific problems. Every step of the Krylov subspace
method such as GMRES method is used in combination with the augmented Lagrangian  preconditioner to solve the  saddle-point
problem $(\ref{saddle})$.
We summarize the implementation of preconditioners 
$\mathcal{P}_{\gamma, \alpha,x}$ and $\mathcal{P}_{\gamma, \alpha,y}$ in
Algorithms $1$ and $2$.
For the linear systems corresponding to $A+ \gamma B^T_{x} Q^{-1} B_{x}$ and $A+ \gamma B^T_{y} Q^{-1} B_{y}$, we distinguish between two approaches:
\begin{itemize}
    \item \textbf{Approach I.} Since $A+ \gamma B^T_{x} Q^{-1} B_{x}$ 
    is  SPD matrix, we solve the linear systems corresponding to this  matrix independently 
by the preconditioned conjugate gradient method (PCG), the matrix is  formed  inside PCG  with incomplete Cholesky preconditioning, ichol(A).

  %  \item \textbf{Approach II.} The matrix $A_{22} + r B^T Q^{-1} B$ is formed explicitly, and PCG with incomplete Cholesky preconditioning was used. We note that while we could successfully compute the \texttt{ichol} factor without diagonal shifts for the two smallest problem sizes, adding the shift 0.01 was found to be necessary for larger sizes. We further note that with this approach we can use larger values of $r$, leading to faster FGMRES convergence.
\end{itemize}
As a result, we summarize the implementation of the preconditioners $\mathcal{P}_{\gamma, \alpha,x}$ and $\mathcal{P}_{\gamma, \alpha,y}$ 
in the form of the following algorithms:
 \begin{algorithm}[H]
 \label{algo1}
%\caption{: NASSP}
\begin{algorithmic}

\caption{: Computation of $(x; y;z) = \mathcal{P}_{\gamma, \alpha,x}^{-1} (r_1; r_2; r_3 )$}
  \State  Step 1. Solve  $z=\alpha Q^{-1} r_3$; where $Q$ is a diagonal matrix;
 \State Step 2. Solve $(A+ \gamma B^T_{x} Q^{-1} B_{x})x = r_1-B_{x}^{T}z$ for $x$;
  \State Step 3. Solve $(A+ \gamma B^T_{x} Q^{-1} B_{x})y = r_2-(1-\gamma\alpha^{-1})B_{y}^{T}z$ for $y$.
\end{algorithmic}
\end{algorithm}
\begin{itemize}
\item    The subsystems corresponding to $(A+ \gamma B^T_{x} Q^{-1} B_{x})$ are solved by PCG method. Within the PCG process, we perform sequence of matrix-vector product, first multiplying vectors by $B_{x}$, $Q^{-1}$
and then by $B^{T}_{x}$. 
\end{itemize}
 \begin{algorithm}[H]
%\caption{: NASSP}
\begin{algorithmic}

\label{algo2}
\caption{: Computation of $(x; y;z) = \mathcal{P}_{\gamma, \alpha,y}^{-1} (r_1; r_2; r_3 )$}
  \State  Step 1. Solve  $z=\alpha Q^{-1}r_3$;
 \State Step 2. Solve $(A+ \gamma B^T_{y} Q^{-1} B_{y})x = r_1-B_{x}^{T}z$ for $x$;
  \State Step 3. Solve $(A+ \gamma B^T_{y} Q^{-1} B_{y})y = r_2-(1-\gamma\alpha^{-1})B_{y}^{T}z$ for $y$.
\end{algorithmic}
\end{algorithm}
We use  the steps described  in Algorithm $1$ to implement Algorithm $2$.
\begin{itemize}
    \item \textbf{Approach II.}
    In step $2$ and $3$ of Algorithms $1$ and $2$,
the secondary objective of this work is not to solve it independently, but instead to utilize  $\mathcal{P}$GCG method~\cite{GBICG} for  solving  linear system with several right-hand sides of the following  form :
\begin{eqnarray}
\label{PSMR1}
(A+ \gamma B^T_{y} Q^{-1} B_{y})\mathcal{X}=\mathcal{H},
\end{eqnarray}
where: $\mathcal{X}$ and $\mathcal{H}$  are both  an $n\times2$ matrices.  Each column of matrix $\mathcal{X}$ is denoted as $\mathcal{X}^{(1)}=x$  and 
$\mathcal{X}^{(2)}=y$, each column of matrix $\mathcal{H}$ is denoted as  $\mathcal{H}^{(1)}=r_1-B_{x}^{T}z$ and $\mathcal{H}^{(2)}=r_2-(1-\gamma\alpha^{-1})B_{y}^{T}z$,
$\mathcal{X}_{0}$ is the initial guess of  solution (\ref{PSMR1}) and $R_{0}=\mathcal{H}-(A+ \gamma B^T_{y} Q^{-1} B_{y})\mathcal{X}_{0}$ is the initial residual.
\end{itemize}
By leveraging the structure of the augmented based-Lagrangian preconditioners $\mathcal{P}_{\gamma, \alpha, x}$,
$\mathcal{P}_{\gamma, \alpha, y}$ and the approach II, in the rest of the paper, we refer
to the new preconditioners as  $\mathcal{P}_{\gamma, \alpha,x, G}$ and $\mathcal{P}_{\gamma, \alpha, y, G}$ where
\begin{itemize}
    \item 
$\mathcal{P}_{\gamma, \alpha,G,x} $: denotes the global augmented Lagrangian preconditioner in the $x$ direction.
 \item 
$\mathcal{P}_{\gamma, \alpha,G,y} $: denotes  global augmented Lagrangian preconditioner in the $y$ direction.
\end{itemize}
To implement the preconditioners $\mathcal{P}_{\gamma, \alpha,x,G}$ and $\mathcal{P}_{\gamma, \alpha,y,G}$, we
use the following algorithms:
 \begin{algorithm}[H]
 \label{algo1}
%\caption{: NASSP}
\begin{algorithmic}

\caption{: Computation of $(x; y;z) = \mathcal{P}_{\gamma, \alpha,x,G}^{-1} (r_1; r_2; r_3 )$}
  \State  Step 1. Solve  $z=\alpha Q^{-1}r_3$;
 \State Step 2. Solve $(A+ \gamma B^T_{x} Q^{-1} B_{x})\mathcal{X}=\mathcal{H}$ for $\mathcal{X}$.
\end{algorithmic}
\end{algorithm}
\begin{itemize}
\item   The subsystem with multiple right-hand sides corresponding to $(A + \gamma B^T_{x} Q^{-1} B_{x})$ is solved using the $\mathcal{P}$GCG method. During the $\mathcal{P}$GCG process, we carry out a sequence of matrix-vector multiplications, starting with multiplication by $B_{x}$, followed by $Q^{-1}$, and then by $B^T_{x}$.
\end{itemize}

 \begin{algorithm}[H]
  
%\caption{: NASSP}
\begin{algorithmic}

\label{algo2}
\caption{: Computation of $(x; y;z) = \mathcal{P}_{\gamma, \alpha,y,G}^{-1} (r_1; r_2; r_3 )$}
  \State  Step 1. Solve  $z=\alpha Q^{-1}r_3$;
 \State Step 2. Solve $(A+ \gamma B^T_{y} Q^{-1} B_{y})\mathcal{X}=\mathcal{H}$ for $\mathcal{X}$. 
\end{algorithmic}
\end{algorithm}
We  apply a similar approach as in Algorithm $3$ to implement Algorithm $4$.
\section{$2\times2$  Strategy for solving two-by-two linear system (\ref{saddle2})}
In this strategy, we employ a single finite element space to discretize the velocity field and achieve the two-by-two partitioning (\ref{saddle2}).
\subsection{Novel Augmented Lagrangian-based preconditioning and global
techniques:}
The iterative solution of the discrete Stokes equations has attracted considerable attention in recent years. Here we limit ourselves to discussing solution algorithms based on preconditioned Krylov subspace methods. The augmented Lagrangian preconditioner allows to solve iteratively the the discrete Stokes equation in a very limited number of iterations, regardless of the mesh refinement. In  the following one constraint-type preconditioner were proposed for accelerating the convergence of Krylov subspace methods.
First, problem (\ref{saddle2}) is reformulated as the equivalent augmented system $\bar{\mathcal{A}}_{2\times 2} \mathbf{\bar u} = \bar{\mathbf{b}}$, where
\begin{equation}
\bar{\mathcal{A}}_{2\times 2}= 
\begin{bmatrix}
A + \gamma B^T Q^{-1}B &B^T \\
B& 0
\end{bmatrix},
\end{equation}
and $\bar{\mathbf{b}} = ({f}+B^T Q^{-1}g; g)$, with $Q$ being an arbitrary SPD matrix and $\gamma > 0$ a user-defined parameter. Evidently, the linear system of equations 
$\bar{\mathcal{A}}_{2\times 2} \mathbf{\bar u} = \bar{\mathbf{b}}$ is equivalent to ${\mathcal{A}}_{2\times 2}  \mathbf{u} = {\mathbf{b}}$. The question of whether the grad–div stabilized discrete
solution is closer or further from the continuous weak solution is out of the scope of this paper. However, several
studies  showed that the grad–div stabilization often improves the mass conservation property
and the velocity error of the discrete solution, for adequate values of $\gamma$.
\subsubsection{ Preconditioning:}\label{Pia}
In this section, to solve the linear system of equations (\ref{saddle}), based on the augmented Lagrangian-based preconditioning. 
The idea of preconditioning is to transform the linear system (\ref{saddle}) into another one that is easier to solve. Left preconditioning of  (\ref{saddle}) gives the following new linear system:
\begin{equation}
\label{aug}
\mathcal{P}^{-1}_{\gamma,\alpha}\bar{\mathcal{A}}_{2\times 2}\mathbf{u} = \mathcal{P}^{-1}\bar{\mathbf{b}},
\end{equation}
where  $\mathcal{P}_{\gamma,\alpha}$ is given as follows
\begin{equation}
\mathcal{P}_{\gamma, \alpha}= 
\begin{bmatrix}
A + \gamma B^T Q^{-1}B &(1-\gamma \alpha^{-1})B^T \\
0& \alpha^{-1}Q
\end{bmatrix}.
\end{equation}
To apply the preconditioner, we need to solve systems of the following form:
 \begin{algorithm}[H]
%\caption{: NASSP}
\begin{algorithmic}

\label{algo2x2}
\caption{: Computation of $(x;y) = \mathcal{P}_{\gamma, \alpha}^{-1} (r_1; r_2 )$}
  \State  Step 1. Solve  $y=\alpha Q^{-1}r_2$;
 \State Step 2. Solve $(A+ \gamma B^T Q^{-1} B)x = r_1-(1- \gamma\alpha^{-1}) B^{T}y$.
\end{algorithmic}
\end{algorithm}
\begin{enumerate}
\item  
We compute $y$,
\item 
The matrix  $A+ \gamma B^T Q^{-1} B$ is SPD, we solve it iteratively by PCG method. We address this by employing the PCG method for iterative solution.  the PCG process, we carry out a sequence of matrix-vector multiplications. First, we multiply the vectors by the matrix 
$B$, then by the inverse of Q, and finally by the transpose of 
B (denoted 
 $B^T$).
\end{enumerate}
\subsection{Spectral analysis}
The distribution of eigenvalues and eigenvectors of a preconditioned matrix has a significant connection to how quickly Krylov subspace methods converge. Hence, it's valuable to analyze the spectral characteristics of the preconditioned matrix, denoted as $\mathcal{P}_{\gamma,\alpha}^{-1}\bar{\mathcal{A}}_{2\times 2}$. In the upcoming theorem, we  will estimate the 
 lower and upper bounds for the eigenvalues of preconditioned
matrix $\mathcal{P}_{\gamma,\alpha}^{-1}\bar{\mathcal{A}}_{2\times 2}$. 
\begin{theorem}
\label{theo1}
Let the preconditioner $\mathcal{P}_{\gamma,\alpha}$ be defined as in~(\ref{Pgamma}). Then the eigenvalues of $\mathcal{P}_{\gamma,\alpha}^{-1}\bar{\mathcal{A}}_{2\times 2}$ are all real, positive and bounded. Furthermore
 the matrix $\mathcal{P}_{\gamma,\alpha}^{-1}\bar{\mathcal{A}}_{2\times 2}$ is diagonalizable and has $n_p+1$ distinct eigenvalues 
 $\{1,\lambda_{1},...,\lambda_{n_p}\}$. 
\end{theorem}
\begin{proof}
Assume that $\lambda$ represents an eigenvalue of the preconditioned matrix  and $\bar{\mathrm{u}}=(\mathrm{u};\mathrm{p})$ is the associated eigenvector. In order to deduce the distribution of eigenvalues, we analyze the following generalized eigenvalue problem
	\begin{eqnarray}
	\label{P-1A}
 \bar{\mathcal{A}}_{2\times 2}\bar{\mathrm{u}} =\lambda\mathcal{P}_{\gamma,\alpha}\bar{\mathrm{u}}.
	\end{eqnarray}
	(\ref{P-1A})~can be reformulated  as follows
 \begin{equation}
	\label{RP1-A}
\left\{
\begin{array}{rl}
(1-\lambda)(A+\gamma B^{T}Q^{-1}B)\mathrm{u}+(1+\lambda(\gamma\alpha^{-1}-1))B^{T}\mathrm{p}&=0, \\
B\mathrm{u}&=-\lambda\alpha^{-1} Q\mathrm{p}.
\end{array}
\right.
\end{equation}
In the case where $\lambda=1$, equation~(\ref{RP1-A}) is always true for $\mathrm{u}\in\text{Null}(B)$, consequently, there exist $n_{u}-n_{p}$ linearly independent eigenvectors $\left(\begin{array}{c}
     \mathrm{u}^{(i)};
 0\end{array}\right)$, $i=1,..,n_{u}-n_{p}$, corresponding to the eigenvalue $1$, where $\mathrm{u}^{(i)}\in\text{Null}(B)$.
     If $\lambda=1$ and $\mathrm{u}=0$, from the second equation of (\ref{RP1-A}), it can be deduced that $\mathrm{p}=0$. This conflicts with the initial assumption that the column vector $\left(
     \mathrm{u};
\mathrm{p}\right)$ is an eigenvector of the preconditioned matrix $\mathcal{P}_{\gamma,\alpha}^{-1}\bar{\mathcal{A}}_{2\times 2}$. If $\lambda\neq 1$ and $\mathrm{p}=0$, from the first equation of $(15)$, it can be deduced that  $\mathrm{u}$ must be $0$. This contradicts
the initial assumption that $(\mathrm{u}; \mathrm{p})$ is the eigenvector of the preconditioned matrix and therefore $\mathrm{u}\neq 0$ and
$\mathrm{p}\neq 0$. 
Since \(\lambda \neq 1\), from $~(\ref{RP1-A})$ we further obtain :
\[
\mathrm{p} =-\frac{\alpha}{\lambda}Q^{-1} B \mathrm{u}.
\]
Substituting $\mathrm{p}$ from the above relation  in the first equation of $(\ref{RP1-A})$, we get :
\begin{eqnarray}
\label{eq1}
    \lambda^{2}(A+\gamma B^T Q^{-1} B)\mathrm{u}-\lambda\left(A+\alpha B^T Q^{-1} B \right)\mathrm{u} +\alpha B^T Q^{-1} B\mathrm{u}  = 0.
\end{eqnarray}
Premultiplying (\ref{eq1}) with $\displaystyle\frac{\mathrm{u}^{T}}{\mathrm{u}^{T}\mathrm{u}}$ 
(\ref{eq1}) gives:
\begin{eqnarray}
\label{eq2}
    (a+\gamma q)\lambda^{2}-\left(a+\alpha q \right)\lambda +\alpha q=0,
\end{eqnarray}
which can be written 
\begin{eqnarray}
\label{eq3}
    \lambda^{2}-b\lambda+c=0,
\end{eqnarray}
where \( a \), \( q \), $b$ and $c$ are given  as follows:
$$a = \frac{\mathrm{u}^{T} A \mathrm{u}}{\mathrm{u}^{T} \mathrm{u}},
q=\frac{\mathrm{u}^{T} B^{T} Q^{-1} B \mathrm{u}}{\mathrm{u}^{T} \mathrm{u}}, b=\frac{a+\alpha q}{a+\gamma q} \hspace{0.1cm}and\hspace{0.1cm} c=\frac{\alpha q}{a+\gamma q}.$$
As a result, it is immediate to see that the roots of  (\ref{eq3})  are real and positive, given by
$n_{p}$ eigenvalues 
$\displaystyle\lambda_1 = \frac{b - \sqrt{b^2 - 4c}}{2}$ and $n_{p}$ eigenvalues 
$\displaystyle\lambda_2 = \frac{b + \sqrt{b^2 - 4c}}{2}$ of the preconditioned matrix.
After some manipulations, $\lambda_1$ and $\lambda_2$ must hold the
following inequalities:
$$
\lambda_1  \geq \frac{2\lambda_{min}(B^{T} Q^{-1} B) }{\lambda_{max}(A)+(1+\alpha-\gamma)\lambda_{max}(B^{T} Q^{-1} B)},\hspace{0.1cm}\lambda_2  \leq \frac{2\alpha\lambda_{max}(B^{T} Q^{-1} B) }{\lambda_{min}(A)+(\alpha-\gamma)\lambda_{min}(B^{T} Q^{-1} B)}.$$

\end{proof}

\section{Numerical results}\label{Numerical}
In this section, we report on the performance of inexact variants of the proposed block preconditioner using a test problem taken from \cite{Elman}, which corresponds to a 2D Stokes flow problem. The programs are performed on a computer with an Intel Core i7-10750H CPU @ 2.60 GHz processor and 16.0 GB RAM using MATLAB R2020b. In all the tables, we report the total required number of outer GMRES iterations and elapsed CPU time (in seconds) under ``Iter'' and "CPU", respectively. The total number of inner GMRES (PCG) iterations to solve subsystems with coefficient matrices $(A+ \gamma B^T_{x} Q^{-1} B_{x})$  and $(A+ \gamma B^T_{y} Q^{-1} B_{y})$   are reported under \(\text{Iter}\) (\(\text{Iter}_{\text{pcg}}\)). No restart is used for either  GMRES iteration. The initial guess is taken to be the zero vector and the iterations are stopped as soon as
\[
\| {b}-\mathcal{A} {x}_k  \|_2 \leq 10^{-7} \|{b} \|_2,
\]
where \( x_k \) is the computed \( k \)-th approximate solution. In the tables, we also include the relative error and relative residual
\[
\text{Err} := \frac{\| x_k - x^* \|_2}{\| x^* \|_2},
\]
and 
\[
\text{Res} := 
\frac{\|{b}-\mathcal{A} {x}_k \|_2} {\|{b} \|_2
},
\]
where \( x^* \) and \( x_k \) are respectively, the exact solution and its approximation obtained in the \( k \)-th iterate. In addition, we have used right-hand sides corresponding to random solution vectors.
\begin{exe}
\label{ex3}
$L$-shaped two dimensional domain $\Omega_{\rightangle}$, parabolic inflow boundary condition, natural outflow boundary condition. Consider the  Stokes  equation system~(\ref{eq:Stokes}) posed in  $\Omega_{\rightangle}=\left( -1,5\right)\times\left( -1,1\right)$. In this scenario, we have a situation where there is a slow flow occurring in a rectangular duct with a sudden expansion. This configuration is often referred to as "flow over a backward facing step". Dirichlet no-flow (zero velocity) boundary conditions on uniform streamline imposed in the inflow boundary $(x=-1; 0\leq y \leq 1)$, the Neumann condition~(\ref{Neumann}) is again applied at the outflow boundary  $(x = 5;-1<y< 1)$.
We use $\mathrm{Q}_{2}-\mathrm{P}_{1}$  mixed finite element approximation from  IFISS library~\cite{Elman} to discretize this problem in  $\Omega_{\rightangle}$, where:
\begin{itemize}
    \item  $Q_{2}$: biquadratic finite element approximation on rectangles for the velocity,
    \item $\mathrm{P}_{1}$: triangular finite element approximation on triangle for the pressure,
\end{itemize}
the nodal positions of this mixed finite element  are illustrated in the following Fig.~\ref{nodal}:
\begin{figure}[H]
\tikzstyle{quadri}=[circle,draw,fill=black,text=white]
\tikzstyle{quadri2}=[circle,draw,fill=white,text=black]
\begin{center}
\begin{tikzpicture}
	\draw[->] (0,0) -- (3,0) node[right] {$\mathcal{L}_{1}$};
	\draw[->] (0,0) -- (0,3) node[above] {$\mathcal{L}_{2}$};
\end{tikzpicture}
\hspace{1cm}
\begin{tikzpicture}
\node[quadri] (E) at(0,4) {};
\node[quadri] (A) at(-2,4) {};
\node[quadri] (Z) at(-4,4) {};
\node[quadri] (W) at(-4,1){};
\node[quadri] (P) at(-2,1){};
\node[quadri] (O) at(0,1){};
\node[quadri2] (H) at(-2,3.2){};
\node[quadri2] (K) at(-1.1,1.8){};
\node[quadri2] (Y) at(-2.7,1.8){};
\node[quadri] (X) at(0,2.5){};
\node[quadri] (I) at(-2,2.5){};
\node[quadri] (J) at(-4,2.5){};
\draw[-,=latex] (E)--(Z);
\draw[-,=latex] (Z)--(W);
\draw[-,=latex] (W)--(O);
\draw[-,=latex] (O)--(E);
\end{tikzpicture}
\end{center}
\caption{}\text{$\mathrm{Q}_{2}-\mathrm{P}_{1}$ element  $\left(\begin{tikzpicture}\node[quadri] (P) at (0,0)  {};\end{tikzpicture}\hspace{0.1cm}\text{ velocity node}; \begin{tikzpicture}\node[quadri2] (Q) at (0,0) {};\end{tikzpicture}\hspace{0.1cm} \text{pressure node} \right),\hspace{0.1cm}\text{local co-ordinate}  \left(\mathcal{L}_{1},\mathcal{L}_{2}\right)$.}
\label{nodal}
\end{figure}
Then we obtain the  nonsingular saddle point problem $(\ref{saddle})$.
\end{exe}
The numerical results of strategy $3\times3$ with approaches I and II for the tested example are listed in Tables $1$ and $3$.  
In Tables $2$ and  $4$, we list numerical results with respect to Iter, CPU and Res in the case of $2\times2$  and $3\times3$ strategies.
\section*{In the case $\gamma=1e-04$ and $\alpha=1e+01$.}
\begin{table}[H]
\begin{center}
\begin{minipage}{\textwidth}
\caption{Results for GMRES in conjunction with preconditioners $\mathcal{P}_{\gamma, \alpha, x}$  and $\mathcal{P}_{\gamma, \alpha, x, G}$ }
\begin{tabular*}{\textwidth}{@{\extracolsep{\fill}}lcccccccc@{\extracolsep{\fill}}}
\toprule
 & \multicolumn{4}{@{}c@{}}{Approach I} &   \multicolumn{4}{@{}c@{}}{Approach II} \\
 \hline
 & \multicolumn{4}{@{}c@{}}{$\mathcal{P}_{\gamma, \alpha, x}$} &   \multicolumn{4}{@{}c@{}}{$\mathcal{P}_{\gamma, \alpha, x, G}$} 
 \\ \cmidrule{1-1}\cmidrule{2-5}\cmidrule{6-9} %
&   \multicolumn{3}{@{}c@{}}{GMRES}&  \multicolumn{1}{@{}c@{}}{ Inner iter}&\multicolumn{3}{@{}c@{}}{GMRES}&  \multicolumn{1}{@{}c@{}}{ Inner iter}\\
\midrule
Size   & $\mathrm{Iter}$ (CPU) & Err& Res & $\mathrm{Iter}_{pcg}$ &  $\mathrm{Iter}$ (CPU) & Err& Res & $\mathrm{Iter}_{pcg}$\\   
   \hline
   1926  &25(0.12) & 7.27e-06 & 1.09e-05& 10 &25(0.07)  &  7.39e-06 & 1.36e-05 & 10
   \\
    \hline
    7302  & 25(1.14) & 1.17e-05  & 1.09e-05& 18 &25(0.62) & 1.17e-05  & 1.09e-05& 18 \\
      \hline
     28420&25(4.70) & 1.19e-05   & 7.93e-06&34&25(2.97) & 1.19e-05   & 7.93e-06 & 34\\
     \hline
%\botrule
\end{tabular*}
\label{Table:2}
\end{minipage}
\end{center}
\end{table}
 It can be seen numerically that the Approach II incorporated with  $\mathcal{P}_{\gamma, \alpha,x,G}$ and $\mathcal{P}_{\gamma, \alpha,y,G}$
preconditioners 
is  more convenient than the Approach I  
incorporated with  $\mathcal{P}_{\gamma, \alpha}$
in terms of
both iteration number and CPU time. As for the number of inner PCG iterations, we observe some differences in the results obtained with Approaches I and II. In the case of Approaches I and II we see an increase in the total number of inner PCG iterations as the size of the problem is increased. With Approach I, the total timings are much higher than that of Approach II. This is due to the fact that solving the two linear systems, with same coefficient matrix $(A+ \gamma B^T_{x} Q^{-1} B_{x})$ or  $(A+ \gamma B^T_{y} Q^{-1} B_{y})$
 leads to a considerably higher CPU time  sparse matrix and also  computing an incomplete Cholesky factorization ichol(A,opts) leads to a considerably expensive PCG iterations.  We conclude that with $\mathcal{P}_{\gamma, \alpha,x,G}$, Approach II is to be preferred to Approach I.
\begin{table}[H]
\begin{center}
\begin{minipage}{\textwidth}
\caption{Results for $2\times2$  and $3\times 3$  strategies in conjunction with preconditioners  $\mathcal{P}_{\gamma, \alpha}$ and $\mathcal{P}_{\gamma, \alpha, x, G}$}
\begin{tabular*}{\textwidth}{@{\extracolsep{\fill}}lcccccccc@{\extracolsep{\fill}}}
\toprule
 & \multicolumn{4}{@{}c@{}}{$2\times2$ 
 Strategy} &   \multicolumn{4}{@{}c@{}}{$3\times3$ 
 Strategy} \\
 \hline
 & \multicolumn{4}{@{}c@{}}{$\mathcal{P}_{\gamma, \alpha}$} &   \multicolumn{4}{@{}c@{}}{$\mathcal{P}_{\gamma, \alpha, x, G}$} 
 \\ \cmidrule{1-1}\cmidrule{2-5}\cmidrule{6-9} %
&   \multicolumn{3}{@{}c@{}}{GMRES}&  \multicolumn{1}{@{}c@{}}{ Inner iter}&\multicolumn{3}{@{}c@{}}{GMRES}&  \multicolumn{1}{@{}c@{}}{ Inner iter}\\
\midrule
Size   & $\mathrm{Iter}$ (CPU) & Err& Res & $\mathrm{Iter}_{pcg}$ &  $\mathrm{Iter}$ (CPU) & Err& Res & $\mathrm{Iter}_{pcg}$\\   
   \hline
   1926  &25(0.09) & 4.15e-04  & 2.03e-04 & 10 &25(0.08) & 7.96e-05  & 1.29e-04& 10
   \\
    \hline
     7302  & 25(0.72) & 1.49e-03 & 3.94e-04  & 18 &25(0.61) & 1.17e-05  & 1.09e-05& 18 \\
      \hline
     28420 &25(6.16) & 4.55e-03   &  6.66e-04&34&25(2.95) & 1.19e-05   & 7.93e-06 & 34\\
     \hline
%\botrule
\end{tabular*}
\label{Table:2}
\end{minipage}
\end{center}
\end{table}
Table $2$ reports the corresponding results of the two strategies with the proposed preconditioners, which show that the  $3\times3$ strategy with $\mathcal{P}_{\gamma, \alpha,x,G}$ perform much better than the $2\times2$ strategy with $\mathcal{P}_{\gamma, \alpha}$, especially for the large
problems. Numerical results are reported in Tables $3$ for the tested methods
with respect to the number of outer iteration steps, inner iteration steps and elapsed CPU time in seconds, denoted
as "Iter", "$Iter_{pcg}$" and "CPU", respectively.
\begin{table}[H]
\begin{center}
\begin{minipage}{\textwidth}
\caption{Results for GMRES in conjunction with preconditioners $\mathcal{P}_{\gamma, \alpha, y}$ and $\mathcal{P}_{\gamma, \alpha, y, G}$}
\begin{tabular*}{\textwidth}{@{\extracolsep{\fill}}lcccccccc@{\extracolsep{\fill}}}
\toprule
 & \multicolumn{4}{@{}c@{}}{Approach I} &   \multicolumn{4}{@{}c@{}}{Approach II} \\
 \hline
 & \multicolumn{4}{@{}c@{}}{$\mathcal{P}_{\gamma, \alpha, y}$} &   \multicolumn{4}{@{}c@{}}{$\mathcal{P}_{\gamma, \alpha, y, G}$} 
 \\ \cmidrule{1-1}\cmidrule{2-5}\cmidrule{6-9} %
 &   \multicolumn{3}{@{}c@{}}{GMRES}&  \multicolumn{1}{@{}c@{}}{ Inner iter}&\multicolumn{3}{@{}c@{}}{GMRES}&  \multicolumn{1}{@{}c@{}}{ Inner iter}\\
\midrule
  Size & $\mathrm{Iter}$ (CPU) & Err& Res & $\mathrm{Iter}_{pcg}$ &  $\mathrm{Iter}$ (CPU) & Err& Res & $\mathrm{Iter}_{pcg}$\\   
   \hline
   1926  &25(0.19) & 7.27e-06 & 1.09e-05& 10 &25(0.08) & 7.96e-05  & 1.29e-04& 10
   \\
    \hline
     7302  & 25(1.11) & 1.17e-05  & 1.09e-05& 18 &25(0.61) & 1.17e-05  & 1.09e-05& 18 \\
      \hline
     28420&25(6.69) & 1.19e-05   & 7.93e-06&34&25(2.95) & 1.19e-05   & 7.93e-06 & 34\\
     \hline
%\botrule
\end{tabular*}
\label{Table:2}
\end{minipage}
\end{center}
\end{table}
The Approach II incorporated with  $\mathcal{P}_{\gamma, \alpha, y, G}$ outperforms the Approach I with  $\mathcal{P}_{\gamma, \alpha, y}$ on efficiency and performance
concerning both iteration steps and CPU times. Moreover, the Approach II incorporated with  $\mathcal{P}_{\gamma, \alpha, y}$
preconditioner is more economical and it is superior to the other two preconditioners regarding
execution time, especially for relatively large size problems.
\begin{table}[H]
\begin{center}
\begin{minipage}{\textwidth}
\caption{Results for  2x2  and 3x3 approaches in conjunction with preconditioners  $\mathcal{P}_{\gamma, \alpha}$ and $\mathcal{P}_{\gamma, \alpha, y, G}$}
\begin{tabular*}{\textwidth}{@{\extracolsep{\fill}}lcccccccc@{\extracolsep{\fill}}}
\toprule
 & \multicolumn{4}{@{}c@{}}{$2\times2$  Strategy} &   \multicolumn{4}{@{}c@{}}{$3\times3$ 
 Strategy} \\
 \hline
 & \multicolumn{4}{@{}c@{}}{$\mathcal{P}_{\gamma, \alpha}$} &   \multicolumn{4}{@{}c@{}}{$\mathcal{P}_{\gamma, \alpha, y, G}$} 
 \\ \cmidrule{1-1}\cmidrule{2-5}\cmidrule{6-9} %
&   \multicolumn{3}{@{}c@{}}{GMRES}&  \multicolumn{1}{@{}c@{}}{ Inner iter}&\multicolumn{3}{@{}c@{}}{GMRES}&  \multicolumn{1}{@{}c@{}}{ Inner iter}\\
\midrule
Size   & $\mathrm{Iter}$ (CPU) & Err& Res & $\mathrm{Iter}_{pcg}$ &  $\mathrm{Iter}$ (CPU) & Err& Res & $\mathrm{Iter}_{pcg}$\\   
   \hline
   1926  &25(0.09) & 4.15e-04 & 2.03e-04& 10 &25(0.08) & 7.96e-05  & 1.29e-04& 10
   \\
    \hline
     7302  & 25(0.72) & 1.49e-03  & 3.49e-04& 18 &25(0.61) & 1.17e-05  & 1.09e-05& 18 \\
      \hline
     28420&25(6.16) & 4.55e-03   & 6.66e-04&34&25(2.95) & 1.19e-05   & 7.93e-06 & 34\\
     \hline
%\botrule
\end{tabular*}
\label{Table:2}
\end{minipage}
\end{center}
\end{table}
\large\textbf{In the case $\gamma=1e-02$ and $\alpha=1e+01$.}
\begin{table}[H]
\begin{center}
\begin{minipage}{\textwidth}
\caption{Results for GMRES in conjunction with preconditioners $\mathcal{P}_{\gamma, \alpha, x}$ and $\mathcal{P}_{\gamma, \alpha, x, G}$}
\begin{tabular*}{\textwidth}{@{\extracolsep{\fill}}lcccccccc@{\extracolsep{\fill}}}
\toprule
 & \multicolumn{4}{@{}c@{}}{Approach I} &   \multicolumn{4}{@{}c@{}}{Approach II} \\
 \hline
& \multicolumn{4}{@{}c@{}}{$\mathcal{P}_{\gamma, \alpha, x}$} &   \multicolumn{4}{@{}c@{}}{$\mathcal{P}_{\gamma, \alpha, x, G}$} 
 \\ \cmidrule{1-1}\cmidrule{2-5}\cmidrule{6-9} %
&   \multicolumn{3}{@{}c@{}}{GMRES}&  \multicolumn{1}{@{}c@{}}{ Inner iter}&\multicolumn{3}{@{}c@{}}{GMRES}&  \multicolumn{1}{@{}c@{}}{ Inner iter}\\
\midrule
Size   & $\mathrm{Iter}$ (CPU) & Err& Res & $\mathrm{Iter}_{pcg}$ &  $\mathrm{Iter}$ (CPU) & Err& Res & $\mathrm{Iter}_{pcg}$\\   
   \hline
   1926  &29(0.13) &  7.51e-06 & 1.79e-04& 9 &29(0.09)  & 7.83e-06 & 1.71e-04 & 9
   \\
    \hline
     7302  & 31(0.84) & 2.40e-05  & 4.26e-04  & 18 &31(0.63) &  2.46e-05   & 4.21e-04 & 18 \\
      \hline
       28420&31(5.13) & 2.82e-05  &4.28e-04 &31&31(3.64) & 2.89e-05   & 4.01e-04  & 31\\
     \hline
%\botrule
\end{tabular*}
\label{Table:2}
\end{minipage}
\end{center}
\end{table}
\begin{table}[H]
\begin{center}
\begin{minipage}{\textwidth}
\caption{Results for  2x2  and 3x3 strategies in conjunction with preconditioners  $\mathcal{P}_{\gamma, \alpha}$ and $\mathcal{P}_{\gamma, \alpha, x, G}$}
\begin{tabular*}{\textwidth}{@{\extracolsep{\fill}}lcccccccc@{\extracolsep{\fill}}}
\toprule
 & \multicolumn{4}{@{}c@{}}{$2\times2$ 
 Strategy} &   \multicolumn{4}{@{}c@{}}{$3\times3$ 
 Strategy} \\
 \hline
 & \multicolumn{4}{@{}c@{}}{$\mathcal{P}_{\gamma, \alpha}$} &   \multicolumn{4}{@{}c@{}}{$\mathcal{P}_{\gamma, \alpha, x, G}$} 
 \\ \cmidrule{1-1}\cmidrule{2-5}\cmidrule{6-9} %
&   \multicolumn{3}{@{}c@{}}{GMRES}&  \multicolumn{1}{@{}c@{}}{ Inner iter}&\multicolumn{3}{@{}c@{}}{GMRES}&  \multicolumn{1}{@{}c@{}}{ Inner iter}\\
\midrule
Size   & $\mathrm{Iter}$ (CPU) & Err& Res & $\mathrm{Iter}_{pcg}$ &  $\mathrm{Iter}$ (CPU) & Err& Res & $\mathrm{Iter}_{pcg}$\\   
   \hline
   1926  &25(0.19) & 7.27e-06 & 1.09e-05& 10 &29(0.08) & 7.96e-05  & 1.29e-04& 10
   \\
    \hline
     7302  & 25(1.11) & 1.17e-05  & 1.09e-05& 18 &31(0.61) & 1.17e-05  & 1.09e-05& 18 \\
      \hline
     28420&25(6.69) & 1.19e-05   & 7.93e-06&34&31(3.64) & 1.19e-05   & 7.93e-06 & 34\\
     \hline
%\botrule
\end{tabular*}
\label{Table:2}
\end{minipage}
\end{center}
\end{table}

\begin{table}[H]
\begin{center}
\begin{minipage}{\textwidth}
\caption{Results for GMRES in conjunction with preconditioners $\mathcal{P}_{\gamma, \alpha, y}$ and $\mathcal{P}_{\gamma, \alpha, y, G}$}
\begin{tabular*}{\textwidth}{@{\extracolsep{\fill}}lcccccccc@{\extracolsep{\fill}}}
\toprule
 & \multicolumn{4}{@{}c@{}}{Approach I} &   \multicolumn{4}{@{}c@{}}{Approach II} \\
 \hline
 & \multicolumn{4}{@{}c@{}}{$\mathcal{P}_{\gamma, \alpha, y}$} &   \multicolumn{4}{@{}c@{}}{$\mathcal{P}_{\gamma, \alpha, y, G}$} 
 \\ \cmidrule{1-1}\cmidrule{2-5}\cmidrule{6-9} %
&   \multicolumn{3}{@{}c@{}}{GMRES}&  \multicolumn{1}{@{}c@{}}{ Inner iter}&\multicolumn{3}{@{}c@{}}{GMRES}&  \multicolumn{1}{@{}c@{}}{ Inner iter}\\
\midrule
Size   & $\mathrm{Iter}$ (CPU) & Err& Res & $\mathrm{Iter}_{pcg}$ &  $\mathrm{Iter}$ (CPU) & Err& Res & $\mathrm{Iter}_{pcg}$\\   
   \hline
   1926  &29(0.10) &  7.98e-06   & 1.85e-04  & 9 &29(0.07) & 7.64e-05 &  1.28e-04 & 9
   \\
    \hline
     7302  & 31(0.78) &  2.59e-05   & 4.68e-04 & 18 &31(0.57) & 2.67e-05 & 4.71e-04& 18 \\
      \hline
     28420&31(5.45) & 5.30e-05  &1.32e-03 &31&31(3.45) & 3.99e-05   & 1.12e-03  & 31\\
     \hline
%\botrule
\end{tabular*}
\label{Table:2}
\end{minipage}
\end{center}
\end{table}
\begin{table}[H]
\begin{center}
\begin{minipage}{\textwidth}
\caption{Results for  2x2  and 3x3 strategies in conjunction with preconditioners  $\mathcal{P}_{\gamma, \alpha}$ and $\mathcal{P}_{\gamma, \alpha, y, G}$}
\begin{tabular*}{\textwidth}{@{\extracolsep{\fill}}lcccccccc@{\extracolsep{\fill}}}
\toprule

 & \multicolumn{4}{@{}c@{}}{$2\times2$  Strategy} &   \multicolumn{4}{@{}c@{}}{$3\times3$ 
 Strategy} \\
 \hline
 & \multicolumn{4}{@{}c@{}}{$\mathcal{P}_{\gamma, \alpha}$} &   \multicolumn{4}{@{}c@{}}{$\mathcal{P}_{\gamma, \alpha, y, G}$} 
 \\ \cmidrule{1-1}\cmidrule{2-5}\cmidrule{6-9} %
&   \multicolumn{3}{@{}c@{}}{GMRES}&  \multicolumn{1}{@{}c@{}}{ Inner iter}&\multicolumn{3}{@{}c@{}}{GMRES}&  \multicolumn{1}{@{}c@{}}{ Inner iter}\\
\midrule
Size   & $\mathrm{Iter}$ (CPU) & Err& Res & $\mathrm{Iter}_{pcg}$ &  $\mathrm{Iter}$ (CPU) & Err& Res & $\mathrm{Iter}_{pcg}$\\   
   \hline
   1926  &29(0.19) & 7.27e-06 & 1.09e-05& 10 &29(0.08) & 7.96e-05  & 1.29e-04& 10
   \\
    \hline
     7302  & 31(1.11) & 1.17e-05  & 1.09e-05& 18 &31(0.61) & 1.17e-05  & 1.09e-05& 18 \\
      \hline
     28420&25(6.69) & 1.19e-05   & 7.93e-06&34&31(3.45) & 1.19e-05   & 7.93e-06 & 34\\
     \hline
%\botrule
\end{tabular*}
\label{Table:2}
\end{minipage}
\end{center}
\end{table}
It was observed in all the Tables  that $2\times 2$ and $3\times 3$  strategies with the inexact augmented Lagrangian-based preconditioner  exhibits faster convergence for smaller  values of $\gamma$. However, for large $\gamma$ the total timings increase due to the fact that the condition number of the blocks $(A+ \gamma B^T_{y} Q^{-1} B_{y})$ and $(A+ \gamma B^T_{x} Q^{-1} B_{x})$ goes up as increase. The $3\times 3$  strategy  incorporated with {$\mathcal{P}_{\gamma, \alpha,x,G}$} and {$\mathcal{P}_{\gamma, \alpha,y,G}$}  preconditioners
 demonstrates significantly better performance.
This superiority is observed across various comparisons with $2\times 2$  strategy incorporated with {$\mathcal{P}_{\gamma, \alpha}$}. Moreover, $3\times 3$  strategy consistently requires less CPU time for convergence. Therefore, it can be concluded that the convergence behavior of $3\times 3$ strategy with {$\mathcal{P}_{\gamma, \alpha,x,G}$} and  {$\mathcal{P}_{\gamma, \alpha,y,G}$} outperforms that of other methods. From the tables above, experimentally observed that the performance of  the preconditioners is  sensitive to $\gamma$ when increase and $\alpha$ decrease. in Tables $6$, $7$ and $8$, it is seen that for $\gamma=1e-02$ and $\alpha=1e+01$, the outer iteration count for GMRES remains essentially constant as the problem size is  increased.  The number of inner  iterations increases drastically for the largest problem size, and this because that the matrix $(A+ \gamma B^T_{x} Q^{-1} B_{x})$ becomes ill-conditioned, for largest values of $\gamma$ and small values of $\alpha$.
\begin{exe}
To discretize problem~(\ref{eq:Stokes})  using  Taylor-Hood $\mathrm{Q}_{2}-\mathrm{Q}_{1}$  mixed-finite element approximation in  $\Omega_{\rightangle}$, we utilize the nodal positions of $\mathrm{Q}_{2}-\mathrm{Q}_{1}$  from  IFISS library~\cite{Elman}, where:
\begin{itemize}
    \item $\mathrm{Q}_{1}$: denotes a quadratic finite element approximation on rectangle,
\end{itemize}
and  the nodal positions of $\mathrm{Q}_{2}-\mathrm{Q}_{1}$  are given below in the following Fig.~\ref{nodal2}:
\begin{figure}[H]
\tikzstyle{quadri}=[circle,draw,fill=black,text=black]
\tikzstyle{quadri2}=[circle,draw,fill=white,text=black]
\begin{center}
\begin{tikzpicture}
	\draw[->] (0,0) -- (3,0) node[right] {$\mathcal{L}_{1}$};
	\draw[->] (0,0) -- (0,3) node[above] {$\mathcal{L}_{2}$};
\end{tikzpicture}
\hspace{1cm}
\begin{tikzpicture}
\node[quadri] (E) at(0,4) {};
\node[quadri] (A) at(-2,4) {};
\node[quadri] (Z) at(-4,4) {};
\node[quadri] (W) at(-4,1){ };
\node[quadri] (P) at(-2,1){ };
\node[quadri] (O) at(0,1){};
\node[quadri] (X) at(0,2.5){};
\node[quadri] (I) at(-2,2.5){};
\node[quadri] (J) at(-4,2.5){};
\draw[-,=latex] (E)--(Z);
\draw[-,=latex] (Z)--(W);
\draw[-,=latex] (W)--(O);
\draw[-,=latex] (O)--(E);
\end{tikzpicture}
\begin{tikzpicture}
\node[quadri2] (E) at(0,4) {};
%\node[quadri2] (A) at(-2,4) {};
\node[quadri2] (Z) at(-4,4) {};
\node[quadri2] (W) at(-4,1){ };
%\node[quadri2] (P) at(-2,1){ };
\node[quadri2] (O) at(0,1){};
\draw[-,=latex] (E)--(Z);
\draw[-,=latex] (Z)--(W);
\draw[-,=latex] (W)--(O);
\draw[-,=latex] (O)--(E);
\end{tikzpicture}
\end{center}
\caption{}\text{$\mathrm{Q}_{2}-\mathrm{Q}_{1}$ element $\left(\begin{tikzpicture}\node[quadri] (P) at (0,0) {};\end{tikzpicture}\hspace{0.1cm}\text{velocity node}; \begin{tikzpicture}\node[quadri2] (Q) at (0,0) {};\end{tikzpicture}\hspace{0.1cm} \text{pressure node}\right) \hspace{0.1cm},\hspace{0.1cm}\text{local co-ordinate}  \left(\mathcal{L}_{1},\mathcal{L}_{2}\right).$}
%\caption{}
\label{nodal2}
\end{figure}
Then we derive the nonsingular saddle point problem $(\ref{saddle})$. 
\end{exe}
To further confirm the effectiveness of the $3\times 3$ strategy incorporated with  $\mathcal{P}_{\gamma, \alpha, x, G}$ or  $\mathcal{P}_{\gamma, \alpha, y, G}$ preconditioners, numerical results of the $2\times 2$ and  $3\times 3$ strategies incorporated with  various preconditioners, 
with respect to Iter, $Iter_{pcg}$, CPU, Res and  Err for saddle point problems with different values of $l$,
are reported in the following Tables. 
\section*{In the case $\gamma=1e-04$ and $\alpha=1e+01$.}
\begin{table}[H]
\begin{center}
\begin{minipage}{\textwidth}
\caption{Results for GMRES in conjunction with preconditioners $\mathcal{P}_{\gamma, \alpha, x}$ and $\mathcal{P}_{\gamma, \alpha, x, G}$}
\begin{tabular*}{\textwidth}{@{\extracolsep{\fill}}lcccccccc@{\extracolsep{\fill}}}
\toprule
 & \multicolumn{4}{@{}c@{}}{Approach I} &   \multicolumn{4}{@{}c@{}}{Approach II} \\
 \hline
 & \multicolumn{4}{@{}c@{}}{$\mathcal{P}_{\gamma, \alpha, x}$} &   \multicolumn{4}{@{}c@{}}{$\mathcal{P}_{\gamma, \alpha, x, G}$} 
 \\ \cmidrule{1-1}\cmidrule{2-5}\cmidrule{6-9} %
&   \multicolumn{3}{@{}c@{}}{GMRES}&  \multicolumn{1}{@{}c@{}}{ Inner iter}&\multicolumn{3}{@{}c@{}}{GMRES}&  \multicolumn{1}{@{}c@{}}{ Inner iter}\\
\midrule
Size   & $\mathrm{Iter}$ (CPU) & Err& Res & $\mathrm{Iter}_{pcg}$ &  $\mathrm{Iter}$ (CPU) & Err& Res & $\mathrm{Iter}_{pcg}$\\   
   \hline
   1926  &28(0.12) & 7.27e-06 & 1.09e-05& 10 &28(0.07)  &  7.39e-06 & 1.36e-05 & 10
   \\
    \hline
     7302  & 30(0.87) & 1.17e-05  & 1.09e-05& 18 &30(0.70) & 1.17e-05  & 1.09e-05& 18 \\
      \hline
     28420&31(5.62) & 2.18e-05   & 7.98e-05 &34&31(3.79) & 2.17e-05  & 7.97e-05 & 34\\
     \hline
%\botrule
\end{tabular*}
\label{Table:2}
\end{minipage}
\end{center}
\end{table}

\begin{table}[H]
\begin{center}
\begin{minipage}{\textwidth}
\caption{Results for  2x2  and 3x3 strategies in conjunction with preconditioners  $\mathcal{P}_{\gamma, \alpha}$ and $\mathcal{P}_{\gamma, \alpha, x, G}$}
\begin{tabular*}{\textwidth}{@{\extracolsep{\fill}}lcccccccc@{\extracolsep{\fill}}}
\toprule
 & \multicolumn{4}{@{}c@{}}{$2\times2$ Strategy} &   \multicolumn{4}{@{}c@{}}{$3\times3$  Strategy} \\
 \hline
 & \multicolumn{4}{@{}c@{}}{$\mathcal{P}_{\gamma, \alpha}$} &   \multicolumn{4}{@{}c@{}}{$\mathcal{P}_{\gamma, \alpha, x, G}$} 
 \\ \cmidrule{1-1}\cmidrule{2-5}\cmidrule{6-9} %
&   \multicolumn{3}{@{}c@{}}{GMRES}&  \multicolumn{1}{@{}c@{}}{ Inner iter}&\multicolumn{3}{@{}c@{}}{GMRES}&  \multicolumn{1}{@{}c@{}}{ Inner iter}\\
\midrule
  Size & $\mathrm{Iter}$ (CPU) & Err& Res & $\mathrm{Iter}_{pcg}$ &  $\mathrm{Iter}$ (CPU) & Err& Res & $\mathrm{Iter}_{pcg}$\\   
   \hline
   1926  &28(0.12) &6.27e-04  & 1.71e-04 & 10 &28(0.08) & 7.96e-05  & 1.29e-04& 10
   \\
    \hline
     7302  & 30(1.01) & 1.54e-03  & 3.23e-04 & 18 &30(0.61) & 1.17e-05  & 1.09e-05& 18 \\
      \hline
     28420&31(7.88) & 5.66e-03     & 6.13e-04 &34&31(2.95) & 1.19e-05   & 7.93e-06 & 34\\
     \hline
%\botrule
\end{tabular*}
\label{Table:2}
\end{minipage}
\end{center}
\end{table}

\begin{table}[H]
\begin{center}
\begin{minipage}{\textwidth}
\caption{Results for GMRES in conjunction wit hpreconditioners $\mathcal{P}_{\gamma, \alpha, y}$ and $\mathcal{P}_{\gamma, \alpha, y, G}$}
\begin{tabular*}{\textwidth}{@{\extracolsep{\fill}}lcccccccc@{\extracolsep{\fill}}}
\toprule
 & \multicolumn{4}{@{}c@{}}{Approach I} &   \multicolumn{4}{@{}c@{}}{Approach II} \\
 \hline
 & \multicolumn{4}{@{}c@{}}{$\mathcal{P}_{\gamma, \alpha, y}$} &   \multicolumn{4}{@{}c@{}}{$\mathcal{P}_{\gamma, \alpha, y, G}$} 
 \\ \cmidrule{1-1}\cmidrule{2-5}\cmidrule{6-9} %
&   \multicolumn{3}{@{}c@{}}{GMRES}&  \multicolumn{1}{@{}c@{}}{ Inner iter}&\multicolumn{3}{@{}c@{}}{GMRES}&  \multicolumn{1}{@{}c@{}}{ Inner iter}\\
\midrule
Size   & $\mathrm{Iter}$ (CPU) & Err& Res & $\mathrm{Iter}_{pcg}$ &  $\mathrm{Iter}$ (CPU) & Err& Res & $\mathrm{Iter}_{pcg}$\\   
   \hline
   1926  &28(0.12) & 2.58e-05  &   4.16e-05  & 10 &28(0.07)  &  2.58e-05& 4.14e-05& 10
   \\
    \hline
     7302  & 30(0.89) & 2.34e-05   & 6.44e-05 & 18 &30(0.70) & 1.17e-05  & 1.09e-05& 18 \\
      \hline
     28420&31(6.21) & 2.18e-05   & 7.98e-05&34&31(4.05) & 2.16e-05  & 7.95e-05 & 34\\
     \hline
%\botrule
\end{tabular*}
\label{Table:2}
\end{minipage}
\end{center}
\end{table}
\begin{table}[H]
\begin{center}
\begin{minipage}{\textwidth}
\caption{Results for  2x2  and 3x3  strategies in conjunction with preconditioners  $\mathcal{P}_{\gamma, \alpha}$ and $\mathcal{P}_{\gamma, \alpha, y, G}$}
\begin{tabular*}{\textwidth}{@{\extracolsep{\fill}}lcccccccc@{\extracolsep{\fill}}}
\toprule
 & \multicolumn{4}{@{}c@{}}{$2\times2$  Strategy} &   \multicolumn{4}{@{}c@{}}{$3\times3$  Strategy} \\
 \hline
 & \multicolumn{4}{@{}c@{}}{$\mathcal{P}_{\gamma, \alpha}$} &   \multicolumn{4}{@{}c@{}}{$\mathcal{P}_{\gamma, \alpha, y, G}$} 
 \\ \cmidrule{1-1}\cmidrule{2-5}\cmidrule{6-9} %
&   \multicolumn{3}{@{}c@{}}{GMRES}&  \multicolumn{1}{@{}c@{}}{ Inner iter}&\multicolumn{3}{@{}c@{}}{GMRES}&  \multicolumn{1}{@{}c@{}}{ Inner iter}\\
\midrule
 Size  & $\mathrm{Iter}$ (CPU) & Err& Res & $\mathrm{Iter}_{pcg}$ &  $\mathrm{Iter}$ (CPU) & Err& Res & $\mathrm{Iter}_{pcg}$\\   
   \hline
   1926  &25(0.12) & 6.27e-04 & 1.71e-04& 10 &25(0.08) & 7.96e-05  & 1.29e-04& 10
   \\
    \hline
     7302  & 25(1.01) & 1.54e-03  & 3.23e-04& 18 &25(0.61) & 1.17e-05  & 1.09e-05& 18 \\
      \hline
     28420&25(7.88) & 5.66e-03   & 6.13e-04&34&25(2.95) & 1.19e-05   & 7.93e-06 & 34\\
     \hline
%\botrule
\end{tabular*}
\label{Table:2}
\end{minipage}
\end{center}
\end{table}
It can be observed from Tables $10$ and  $11$, that the  $\mathcal{P}_{\gamma, \alpha, x, G}^{GMRES}$ and $\mathcal{P}_{\gamma, \alpha, y, G}^{GMRES}$ 
 methods  have great advantage in the $\mathrm{CPU}$ compared with $\mathcal{P}_{\gamma, \alpha, x}^{GMRES}$ and $\mathcal{P}_{\gamma, \alpha, y}^{GMRES}$  methods, which shows with Approach II  the total timings are much smaller than in the case  of Approach I.  Although the results in Tables $10$ and $12$ indicates 
  applying $3\times 3$ strategy  and Algorithms $3$ and $4$ to solve the problem with several right-hand sides $(A+ \gamma B^T_{x} Q^{-1} B_{x})\mathcal{X}=\mathcal{H}$ or  $(A+ \gamma B^T_{y} Q^{-1} B_{y})\mathcal{X}=\mathcal{H}$,  need less computing time than using $2\times 2$ strategy with Algorithms $1$ and $2$. 
 \section*{In the case $\gamma=1e-02$ and $\alpha=1e+01$.}
\begin{table}[H]
\begin{center}
\begin{minipage}{\textwidth}
\caption{Results for GMRES in conjunction with preconditioners $\mathcal{P}_{\gamma, \alpha, x}$ and $\mathcal{P}_{\gamma, \alpha, x, G}$}
\begin{tabular*}{\textwidth}{@{\extracolsep{\fill}}lcccccccc@{\extracolsep{\fill}}}
\toprule
 & \multicolumn{4}{@{}c@{}}{Approach I} &   \multicolumn{4}{@{}c@{}}{Approach II} \\
 \hline
 & \multicolumn{4}{@{}c@{}}{$\mathcal{P}_{\gamma, \alpha, x}$} &   \multicolumn{4}{@{}c@{}}{$\mathcal{P}_{\gamma, \alpha, x, G}$} 
 \\ \cmidrule{1-1}\cmidrule{2-5}\cmidrule{6-9} %
&   \multicolumn{3}{@{}c@{}}{GMRES}&  \multicolumn{1}{@{}c@{}}{ Inner iter}&\multicolumn{3}{@{}c@{}}{GMRES}&  \multicolumn{1}{@{}c@{}}{ Inner iter}\\
\midrule
 Size  & $\mathrm{Iter}$ (CPU) & Err& Res & $\mathrm{Iter}_{pcg}$ &  $\mathrm{Iter}$ (CPU) & Err& Res & $\mathrm{Iter}_{pcg}$  \\   
   \hline
   1926  &39(0.11) & 4.96e-03 & 3.97e-02 & 10 &39(0.08)  &  4.96e-03 & 3.97e-02   & 10
   \\
    \hline
     7302  & 40(1.08) & 6.34e-03   & 5.62e-02   & 18 &40(0.81) & 6.34e-03  & 5.62e-02 & 18 \\
      \hline
     28420&43(7.68) & 8.24e-03   & 7.88e-02&34&43(5.05) & 8.24e-03   &  7.88e-02 & 34\\
     \hline
%\botrule
\end{tabular*}
\label{Table:2}
\end{minipage}
\end{center}
\end{table}
\begin{table}[H]
\begin{center}
\begin{minipage}{\textwidth}
\caption{Results for  2x2  and 3x3 strategies in conjunction with preconditioners  $\mathcal{P}_{\gamma, \alpha}$ and $\mathcal{P}_{\gamma, \alpha, x, G}$}
\begin{tabular*}{\textwidth}{@{\extracolsep{\fill}}lcccccccc@{\extracolsep{\fill}}}
\toprule
 & \multicolumn{4}{@{}c@{}}{$2\times2$ Strategy} &   \multicolumn{4}{@{}c@{}}{$3\times3$ Strategy} \\
 \hline
 & \multicolumn{4}{@{}c@{}}{$\mathcal{P}_{\gamma, \alpha}$} &   \multicolumn{4}{@{}c@{}}{$\mathcal{P}_{\gamma, \alpha, x, G}$} 
 \\ \cmidrule{1-1}\cmidrule{2-5}\cmidrule{6-9} %
&   \multicolumn{3}{@{}c@{}}{GMRES}&  \multicolumn{1}{@{}c@{}}{ Inner iter}&\multicolumn{3}{@{}c@{}}{GMRES}&  \multicolumn{1}{@{}c@{}}{ Inner iter}\\
\midrule
  Size & $\mathrm{Iter}$ (CPU) & Err& Res & $\mathrm{Iter}_{pcg}$ &  $\mathrm{Iter}$ (CPU) & Err& Res & $\mathrm{Iter}_{pcg}$\\   
   \hline
   1926 &25(0.19) & 7.27e-06 & 1.09e-05& 10 &25(0.08) & 7.96e-05  & 1.29e-04& 10
   \\
    \hline
     7302  & 25(1.11) & 1.17e-05  & 1.09e-05& 18 &25(0.61) & 1.17e-05  & 1.09e-05& 18 \\
      \hline
     28420&25(6.69) & 1.19e-05   & 7.93e-06&34&25(2.95) & 1.19e-05   & 7.93e-06 & 34\\
     \hline
%\botrule
\end{tabular*}
\label{Table:2}
\end{minipage}
\end{center}
\end{table}

\begin{table}[H]
\begin{center}
\begin{minipage}{\textwidth}
\caption{Results for GMRES in conjunction with preconditioners $\mathcal{P}_{\gamma, \alpha, y}$ and $\mathcal{P}_{\gamma, \alpha, y, G}$}
\begin{tabular*}{\textwidth}{@{\extracolsep{\fill}}lcccccccc@{\extracolsep{\fill}}}
\toprule
 & \multicolumn{4}{@{}c@{}}{Approach I} &   \multicolumn{4}{@{}c@{}}{Approach II} \\
 \hline
 & \multicolumn{4}{@{}c@{}}{$\mathcal{P}_{\gamma, \alpha, y}$} &   \multicolumn{4}{@{}c@{}}{$\mathcal{P}_{\gamma, \alpha, y, G}$} 
 \\ \cmidrule{1-1}\cmidrule{2-5}\cmidrule{6-9} %
&   \multicolumn{3}{@{}c@{}}{GMRES}&  \multicolumn{1}{@{}c@{}}{ Inner iter}&\multicolumn{3}{@{}c@{}}{GMRES}&  \multicolumn{1}{@{}c@{}}{ Inner iter}\\
\midrule
Size   & $\mathrm{Iter}$ (CPU) & Err& Res & $\mathrm{Iter}_{pcg}$ &  $\mathrm{Iter}$ (CPU) & Err& Res & $\mathrm{Iter}_{pcg}$\\   
   \hline
   1926  &37(0.13) &  4.94e-03 & 3.96e-02 & 10 &37(0.09) & 4.94e-03  & 3.96e-02& 10
   \\
    \hline
     7302  & 38(1.03) & 6.30e-03  & 5.60e-02    & 18 &38(0.80) & 6.30e-03  & 5.60e-02 & 18 \\
      \hline
     28420&39(7.19) &  8.20e-03    &  7.85e-02 &33&39(4.70) & 8.20e-03  & 7.85e-02 & 33\\
     \hline
%\botrule
\end{tabular*}
\label{Table:2}
\end{minipage}
\end{center}
\end{table}
\begin{table}[H]
\begin{center}
\begin{minipage}{\textwidth}
\caption{Results for  $2x2$  and 3x3 strategies in conjunction with preconditioners  $\mathcal{P}_{\gamma, \alpha}$ and $\mathcal{P}_{\gamma, \alpha, y, G}$}
\begin{tabular*}{\textwidth}{@{\extracolsep{\fill}}lcccccccc@{\extracolsep{\fill}}}
\toprule
 & \multicolumn{4}{@{}c@{}}{$2\times2$ Strategy} &   \multicolumn{4}{@{}c@{}}{$3\times3$ Strategy} \\
 \hline
 & \multicolumn{4}{@{}c@{}}{$\mathcal{P}_{\gamma, \alpha}$} &   \multicolumn{4}{@{}c@{}}{$\mathcal{P}_{\gamma, \alpha, y, G}$} 
 \\ \cmidrule{1-1}\cmidrule{2-5}\cmidrule{6-9} %
&   \multicolumn{3}{@{}c@{}}{GMRES}&  \multicolumn{1}{@{}c@{}}{ Inner iter}&\multicolumn{3}{@{}c@{}}{GMRES}&  \multicolumn{1}{@{}c@{}}{ Inner iter}\\
\midrule
  Size & $\mathrm{Iter}$ (CPU) & Err& Res & $\mathrm{Iter}_{pcg}$ &  $\mathrm{Iter}$ (CPU) & Err& Res & $\mathrm{Iter}_{pcg}$\\   
   \hline
   1926  &25(0.19) & 7.27e-06 & 1.09e-05& 10 &25(0.08) & 7.96e-05  & 1.29e-04& 10
   \\
    \hline
     7302  & 25(1.11) & 1.17e-05  & 1.09e-05& 18 &25(0.61) & 1.17e-05  & 1.09e-05& 18 \\
      \hline
     28420&25(6.69) & 1.19e-05   & 7.93e-06&34&25(2.95) & 1.19e-05   & 7.93e-06 & 34\\
     \hline
%\botrule
\end{tabular*}
\label{Table:2}
\end{minipage}
\end{center}
\end{table}
By comparing the results in Tables $13$, $14$, $15$ and $16$ it can be seen that our proposed strategy  incorporated with the 
preconditioned $\mathcal{P}_{\gamma, \alpha, x, G}^{GMRES}$ and $\mathcal{P}_{\gamma, \alpha, y, G}^{GMRES}$
 methods succeed in producing high-quality approximate solutions in
all cases, while the
 $3\times 3$ strategy incorporated with 
preconditioned $\mathcal{P}_{\gamma, \alpha,x,G}^{GMRES}$, $\mathcal{P}_{\gamma, \alpha,y,G}^{GMRES}$
methods, outperforms the classical $2\times 2$ strategy incorporated with 
preconditioned $\mathcal{P}_{\gamma, \alpha}^{GMRES}$ method, 
in terms of Iter and CPU times. Besides, numerical results in Tables  above show that the $3\times 3$ strategy incorporated with 
preconditioned $\mathcal{P}_{\gamma, \alpha,x,G}^{GMRES}$ and 
$\mathcal{P}_{\gamma, \alpha,y,G}^{GMRES}$
methods with
proper $\alpha$ and $\gamma$ is still very efficient  even for larger  size of problems.
\section{Conclusion}\label{conclusion}
In this paper, we  introduce a new class of augmented Lagrangian-preconditioners  based on global conjugate gradient (GCG) method for solving three-by-three linear systems, focusing on systems arising from finite element discretizations of the Stokes flow problem. Numerical experiments on a challenging $2D$ model problem indicate that the corresponding inexact preconditioner with $3\times3$ strategy can achieve significantly faster convergence compared to previous versions of the augmented Lagrangian-based preconditioner.
Future work will concentrate on replacing the incomplete Cholesky inner preconditioners with multilevel preconditioners to enhance the scalability of the global conjugate gradient and  needs to find an optimal
parameter to realize the fast convergence rate.
\bibliographystyle{elsarticle-num-names}

\end{document}